\documentclass[twoside,11pt]{article}%
\usepackage{latexsym}
\usepackage{amsmath,amssymb,epsfig,subfigure}
\usepackage{amsthm,enumerate,verbatim}
\usepackage{amsfonts}
\usepackage{color}
\usepackage{amsmath}
\usepackage{graphicx}
\usepackage{epstopdf}
\usepackage{cite}
\usepackage{algpseudocode}
\usepackage{algorithmicx,algorithm}

\providecommand{\U}[1]{\protect\rule{.1in}{.1in}}
\providecommand{\U}[1]{\protect\rule{.1in}{.1in}}
\newcommand{\BE}{\begin{equation}}
\newcommand{\EE}{\end{equation}}

\newcommand{\zd}{\,\mathrm{d}}
\numberwithin{equation}{section}
\newtheorem{proposition}{Proposition}[section]
\newtheorem{theorem}[proposition]{Theorem}
\newtheorem{lemma}[proposition]{Lemma}

\newtheorem{example}[proposition]{Example}

\def\dfrac{\displaystyle\frac}

\begin{document}

 \title{{\bf A high-order and fast scheme with variable time steps for the time-fractional Black-Scholes equation}}
 
  \author{Kerui Song\thanks{Email: 925281694@qq.com. School of Economic Mathematics, Southwestern University of Finance and Economics, Chengdu, China.}
 \and Pin Lyu\thanks{Corresponding author.
 Email: plyu@swufe.edu.cn. School of Economic Mathematics, Southwestern University of Finance and Economics, Chengdu, China. This author is supported by the National Natural Science Foundation of China (12101510) and the Fundamental Research Funds for the Central Universities (JBK2102010).}
}
\date{}
 \maketitle\normalsize

\begin{abstract}
In this paper, a high-order and fast numerical method is investigated for the time-fractional Black-Scholes equation. In order to deal with the typical weak initial singularities of the solution, we construct a finite difference scheme with variable time steps, where the fractional derivative is approximated by the nonuniform Alikhanov formula  and the sum-of-exponentials (SOE) technique. In the spatial direction, an average approximation with fourth-order accuracy is employed. The stability and the convergence with second-order in time and fourth-order in space of the proposed scheme are religiously derived by the energy method. Numerical examples are given to demonstrate the theoretical statement.
\end{abstract}

\section{Introduction}
In recent years, the option theory has been widely used in financial and economic fields, so the study of option pricing becomes more important in both theoretical significance and practical application. The Black-Scholes model, a second-order parabolic partial differential equation related to stock price and time, is used for pricing European or American put and call options on stock \cite{Black and Scholes}. 

With the proposals of the fractional partial differential equation about stochastic model and financial theory, a growing number of scholars began to study fractional option pricing model and made great progress. Wyss \cite{Wyss} considered the pricing of option derivatives under a time-fractional Black-Scholes equation preliminarily by replacing the time first-order derivative by a fractional derivative of order $\alpha~(0<\alpha\leq 1)$, and derived a closed-form solution for European vanilla options. Cartea and del-Castillo-Negrete \cite{Castillo-Negrete} displayed that some particular L\'{e}vy processes satisfy a fractional partial differential equation,  and employed numerical methods to solve the related fractional models in order to  price exotic options, in particular barrier options.  Jumarie \cite{Jumarie1, Jumarie2} applied the fractional Taylor formula to remove the effects of the non-zero initial value of the function. Under the It\^{o} lemma of fractional order illustrated in the special case of a fractional growth with white noise, they derived the time and space fractional Black-Scholes equations. By assuming that the stock price dynamics follows a fractional It\^{o} process, Liang et al. \cite{Liang1,Liang2} proposed a bi-fractional Black-Merton-Scholes model of option pricing. A slightly simplified version based on Liang et al.'s model \cite{Liang2} was then discussed in Chen et al. \cite{Chen}, they assumed that the underlying asset price is consistent with the classical Brownian motion, and the spatial-fractional derivative in the governing equation disappears, but the time-fractional derivative remains.

Since the analytical solution of the fractional differential equations are always hard to find, it is necessary to study efficient numerical methods for the related problems.  In this paper, we will discuss a high-order finite difference method for the time-factional Black-Scholes equation \cite{Chen}:
\begin{equation}\label{eq1}
\begin{array}{l} \vspace{2mm}
\displaystyle {\frac{\partial^\alpha C}{\partial \zeta^\alpha}}+\frac{1}{2}{\varrho^{2}}S^{2}\frac{\partial^{2}C}{\partial S^{2}}+(r-D)S\frac{\partial C}{\partial S}-rC=0, \quad (S,\tau)\in \tilde{\Omega}\times [0,T),\\
C(S,T)=R(S), \qquad S_l<S<S_r,\\
C(S_l,\zeta)=P(\zeta), ~C(S_r,\zeta)=Q(\zeta), \qquad {\zeta}\in [0,T),
\end{array}
\end{equation}
where $C(S,\tau)$ is the time-$\tau$ price of a European-style double barrier option with the underlying $S$, $\tilde{\Omega} = (S_l, S_r)\subset\mathbb{R}^{+}$, $\zeta$ is the current time, $T$ is the expiry, $r$ is the risk-free interest rate, $D$ is the dividend yield and ${\varrho}$ is the volatility of the returns. The functions $P$ and $Q$ are the rebates paid when the corresponding barrier is hit, and $R$ is the payoff of the option. The time derivative in \eqref{eq1} is defined as
$$
\frac{\partial^\alpha C}{\partial \zeta^\alpha}=\int_{\zeta}^T{\omega_{1-\alpha}}(\eta-\zeta){\partial}_\eta C(S,\eta)d\eta  \quad 0<{\alpha}<1,
$$
where the kernel ${\omega_{\beta}}(t):={t^{\beta-1}}/{\Gamma(\beta)}$,  $t>0$.

 We notice that (see also \cite{Staelen-CMA}) by taking the auxiliary variables: $x=\ln S$, $t=T-\zeta$ and the function  $w(x,t)=C(e^x,T-t)$, one has
\begin{align}\label{Caputofrac}
-\frac{\partial^\alpha C}{\partial \zeta^\alpha}=\int_{0}^t{\omega_{1-\alpha}}(t-s){\partial}_sw(x,s)ds= \mathcal{D}_t^{\alpha}w,
\end{align}
where $\mathcal{D}_t^{\alpha}$ represents the Caputo derivative of order $\alpha\in(0,1)$. Then, the problem \eqref{eq1} can be transformed to the following equations with constant coefficients:
\begin{equation}\label{Ordianary}
\begin{array}{l} \vspace{2mm}
\displaystyle \mathcal{D}_t^{\alpha}w-a\frac{\partial^{2}w}{\partial x^{2}}-b\frac{\partial w}{\partial x}+cw=0, \quad (x,t)\in \Omega\times (0,T],\\
w(x,0)= r(x), \qquad x\in\Omega,\\
w(x_l,t)=p(t), ~w(x_r,t)=q(t),  \qquad t\in (0,T],
\end{array}
\end{equation}
where $a=\frac{1}{2}\varrho^2$, $b=r-a-D$, $c=r$ and $\Omega=(x_l,x_r)$.

Moreover, denote $u(x,t):=w(x,t)-z(x,t)$, where
$$\displaystyle z(x,t):=\frac{q(t)-p(t)}{x_r-x_l}(x-x_l)+p(t).$$
It is easy to see that the problem \eqref{Ordianary} is equivalent to the next equations with homogeneous boundary conditions:
\begin{equation}\label{GoverningEq}
\begin{array}{l} \vspace{2mm}
\displaystyle
 \mathcal{D}_t^{\alpha}u= a\frac{\partial^{2}u}{\partial x^{2}}+b\frac{\partial u}{\partial x}-cu+f(x,t), \quad (x,t)\in \Omega\times (0,T],\\
u(x,0)= \varphi(x), \qquad x\in\Omega,\\
u(x,t)=0,  \qquad (x,t)\in\partial\Omega\times\in (0,T],
\end{array}
\end{equation}
where
\begin{align*}
&f(x,t)=b\frac{q(t)-p(t)}{x_r-x_l}-cz-\mathcal{D}_t^{\alpha}z,\\
& \varphi(x)=r(x)-\frac{q(0)-p(0)}{x_r-x_l}(x-x_l)-p(0).
\end{align*}

In recent years, several numerical methods for solving fractional Black-Scholes model have been developed.  In \cite{Cen-CMA2018}, a difference scheme on nonuniform time grids is proposed for an equivalent integral-differential equation of the problem \eqref{eq1}, but it is only first-order convergent in time. For the problem \eqref{GoverningEq}, Zhang et al. \cite{Zhang-Liu} discussed a discrete implicit numerical scheme which has the temporal $(2-\alpha)$-order and spatial second-order convergence. Roul \cite{Roul-AML} studied a finite difference method with the convergence of ($2-\alpha$)-order in time and fourth-order in space. De Staelen et al. \cite{Staelen-CMA} investigated an implicit numerical scheme with a temporal accuracy of ($2-\alpha$)-order and spatial accuracy of fourth-order by using the Fourier analysis method. It should be noticed that all of the above numerical methods are based on the analytical solution is smooth enough in the time direction.
However, the solution of time-fractional differential equations generally exhibits weak singularities near the initial time, which such that most of the classical numerical methods based on smooth assumptions are difficult to achieve the high-order convergence in the general situations, one may refer to \cite{Stynes-SIAM2017,Jin-IMA2016} for the discussion on the regularity of the solution of time-fractional diffusion equations and the restrictions of some classical approximations based on sufficient smooth solutions. 

To deal with the weak singularities of the solutions, a natural and efficient way is implementing numerical methods with variable step sizes (the mesh will be nonuniform), that is concentrating more mesh points around the (weak) singular points to catch the rapid variation of the solution and use large steps while the solution changes slowly. Numerical methods with variable time step sizes are found to be very efficient and fairly popular in recent years to solve the weak initial singularities of the time-fractional partial differential equations \cite{ChenStynesJSC2019,KoptevaMC2019,Sharperror,Liao-2order,second-fast,LiaoYanZhang2018,LyuVong2020diffu-wave,LyuVong2021wave-variable,Stynes-SIAM2017,WangANM2021}.
In view of the practical advantage of the nonuniform mesh technique, we will discuss the Alikhanov formula with variable steps to develop an efficient finite difference scheme with second-order temporal accuracy for the time-fractional Black-Scholes equation with weak singular solutions, and the sum-of-exponentials (SOE) technique \cite{Fast-evaluation} will also be utilized at the same time to the discrete Caputo derivative to save the computation costs. Moreover, a high-order average approximation will be employed to approximate the space derivatives to such that the proposed fully discrete scheme is fourth-order accuracy in the spatial direction.
The stability of the proposed scheme will be established according to the analysis framework developed in \cite{Adiscrete,Liao-2order} and some matrix analysis techniques. Based on the following regularity assumptions on the exact solution $u$ (for $0<t\leq T$):
\begin{align}\label{regularity1}
 &\left\|\frac{\partial^{k+l} u}{\partial t^k\partial x^l}\right\|_{L^\infty}\leq C(1+t^{\sigma-l}),\quad \mbox{for}\quad k=0,1,2,3,\quad l=0,1,2;\\\label{regularity2}
  &\left\|\frac{\partial^{m} u}{\partial x^m}\right\|_{L^\infty}\leq C, \quad \mbox{for}\quad m=3,4,5,6,
\end{align}
where $\sigma\in (0,1)\cup (1,2)$ is a regularity parameter, and under weak mesh restrictions,
we can show that the proposed nonuniform scheme is unconditionally convergent with second-order accuracy in time and fourth-order accuracy in space.

The rest of the paper is organized as follows. In Section \ref{Space-time-approx}, we introduce a spatial fourth-order approximation for the governing problem, and show the fast nonuniform Alikhanov formula and derive some necessary properties of the discrete coefficients. In Section \ref{NumericalScheme}, based on the fast  nonuniform Alikhanov formula and the spatial forth-order approximation, we construct an efficient nonuniform finite difference scheme for the time-fractional Black-Scholes equation. The unconditional stability and the convergence of second-order in time and fourth-order in space for the proposed scheme are well displayed by energy method. Numerical examples are provided in Section \ref{numerical} to demonstrate the theoretical statement. A brief conclusion is followed in Section \ref{conclusion}.

\section{The high-order and nonuniform approximations}\label{Space-time-approx}
\subsection{Spatial high-order approximation}\label{Spatial-high-order}

Some notations are needed. For a positive integer $M$, the spatial step size $h=(x_r-x_l)/M$, the discrete grid ${\Omega}_h:=\{x_l+ih~|~1\leq i\leq M-1\}$ and $\bar\Omega_h:={\Omega}_h\cup\partial\Omega$. Denote the space of grid functions $\mathcal{V}_h:= \{v_i~|~v_i~\mbox{vanishes on}~ \partial\Omega_h, ~0\leq i\leq M\}$. For two grid functions $v_i, w_i \in \mathcal{V}_h$, the inner product is denoted as $\displaystyle \langle v,w\rangle := h\sum_{i=1}^{M-1}v_iw_i$, and the discrete $L^2$ norm is $\|v\| := \sqrt{\langle v,v\rangle}$. Define spatial central difference operators $\delta_x^2 v_i:=(v_{i+1}-2v_i+v_{i-1})/h^2$ and $\delta_{\hat x} v_i:=(v_{i+1}-v_{i-1})/(2h)$.

In order to obtain a spatial high accuracy numerical scheme, we will utilize a fourth-order approximation which is derived in \cite{Staelen-CMA} to discretize the space derivatives of the time-fractional Black-Scholes equations \eqref{GoverningEq}, we review it briefly in the following.

Applying the Taylor formula at the grid points $x_i~(1\leq i\leq M-1)$, and based on the assumption \eqref{regularity2}, we have
\begin{align}\label{Taylor-expansion1}
& \frac{\partial u(x_i, t)}{\partial x}=\delta_{\hat x}u(x_{i}, t)-\frac{h^2}{6}\frac{\partial^{3}u(x_i, t)}{\partial x^{3}}+\mathcal{O}(h^4),\\\label{Taylor-expansion2}
& \frac{\partial^{2}u(x_i, t)}{\partial x^{2}}=\delta_x^2 u(x_{i}, t)-\frac{h^2}{12}\frac{\partial^{4}u(x_i, t)}{\partial x^{4}}+\mathcal{O}(h^4).
\end{align}
Denote $g(x,t):=\mathcal{D}_t^{\alpha}u(x,t)+cu(x,t)-f(x,t)$.
It follows from  the first equation in \eqref{GoverningEq}, and \eqref{Taylor-expansion1}--\eqref{Taylor-expansion2} that
\begin{align}\label{high-order1}
\displaystyle a\delta_x^2{u(x_i,t)}+b\delta_{\hat{x}}u(x_i,t)-\frac{h^2}{12}\left(a\frac{\partial^{4}u(x_i, t)}{\partial x^{4}}+2b\frac{\partial^{3}u(x_i, t)}{\partial x^{3}}\right)+\mathcal{O}(h^4)=g(x_i,t),
\end{align}
On the other hand, suppose $g(x,\cdot)\in {\cal C}^2(\Omega)$, the Taylor formula shows that
\begin{align}\label{Taylor-expansion3}
 \frac{\partial^{3}u(x_i, t)}{\partial x^{3}}&=\frac{1}{a}\left(\delta_{\hat{x}}g(x_i,t)-b\delta_x^2{u(x_i,t)}\right)+\mathcal{O}(h^2),\\\label{Taylor-expansion4}
 \frac{\partial^{4}u(x_i, t)}{\partial x^{4}}&=\frac{1}{a}\left(\delta_x^2{g(x_i,t)}-\frac{b}{a}\left(\delta_{\hat{x}}g(x_i,t)-b\delta_x^2{u(x_i,t)}\right)\right)+\mathcal{O}(h^2).
\end{align}
Substituting \eqref{Taylor-expansion3}--\eqref{Taylor-expansion4} in \eqref{high-order1}, one has
\begin{align}\label{high-order-app}
 \frac{h^2}{12}\left(\delta_x^2{g(x_i,t)}
+\frac{b}{a}\delta_{\hat{x}}g(x_i,t)\right)+g(x_i,t)=\left(a+\frac{h^2b^2}{12a}\right)\delta_x^2{u(x_i,t)}+b\delta_{\hat{x}}u(x_i,t)+\mathcal{O}(h^4).
\end{align}
Thus, from \eqref{high-order-app}, we have a high-order operator ${\cal H}:=\dfrac{h^2}{12}\left(\delta_x^2+\dfrac{b}{a}\delta_{\hat{x}}\right)+1$ to implement a spatial fourth-order accurate approximation.

\subsection{Fast nonuniform Alikhanov formula}

Our numerical method will be implemented on possible nonuniform time partitions: $0=t_0<t_1<t_2<{\cdots}<t_N=T$, where $N$ is a positive integer. Denote a fractional time level $t_{n-{\theta}}:={\theta}t_{n-1}+(1-\theta)t_n$ for an off-set parameter ${\theta}={\alpha}/{2}$, and take ${\tau}_k:=t_k-t_{k-1}$ ($1\leq k \leq N$) as the $k$th time-step size, and ${\tau}:=\mathop{\mbox{max}}\limits_{1\leq k \leq N-1}{\tau}_k$ being the maximum step size. Besides, the local step-size ratios are defined as
$${\rho}_k:=\frac{\tau_k}{\tau_{k+1}} ~\mbox{ for }~  1\leq k \leq N-1, \quad \mbox{and}  \quad {\rho}:=\max\limits_{1\leq k \leq N-1}{\rho}_k.$$
The numerical analysis of our proposed scheme will be based on the following weak assumptions on the temporal mesh:
\begin{itemize}
\item[{\bf M1.}] The maximum time-step ratio is ${\rho}=7/4$.
\item[{\bf M2.}]  There is a constant $C_\gamma>0$ such that $\tau_k\leq C_\gamma\tau\mbox{min}\{1,t_k^{1-1/\gamma}\}$ for $1\leq k\leq N$, with $t_k\leq C_\gamma t_{k-1}$ and $\tau_k/t_k\leq C_\gamma\tau_{k-1}/t_{k-1}$ for $2\leq k\leq N$.
\end{itemize}

We next introduce the time approximation for the Caputo derivative. For any time sequence $(v^{k})_{k=0}^N$, define the backward difference ${\nabla}_\tau v^k:=v^k-v^{k-1}$ and the interpolated value $v^{n-\theta}:={\theta}v^{n-1}+(1-\theta)v^n$. Denoting $\Pi_{1,n}v$ the linear interpolation of a function $v$ with respect to the nodes $t_{k-1}$ and $t_k$, and $\Pi_{2,n}v$ the quadratic interpolation of a function $v$ with respect to the nodes $t_{k-1}$, $t_k$  and $t_{k+1}$.
To obtain a second-order scheme, we apply the Alikhanov formula on possible nonuniform meshes \cite{Liao-2order} to approximate the Caputo derivative. Meanwhile, the SOE technique is employed to result a nonuniform and fast Alikhanov formula in order to reduce the computational costs.



 First of all, we review the SOE approximation (see also \cite[Theorem 2.5]{Fast-evaluation}  or \cite[Lemma 5.1]{second-fast}) which is designed for the kernel function $\omega_{1-\alpha}(t)$ on the interval $[\Delta t, T]$:
\begin{lemma}\label{FA}
	For the given $\alpha\in (0,1)$, an absolute tolerance error $\epsilon\ll1$, a cut-off time $\Delta t>0$ and a finial time $T$, there exists a
	positive integer $N_q$, positive quadrature nodes $s^l$ and corresponding positive weights $\varpi^l (1\leq l\leq N_q)$ such that
	\begin{align*}
	\left| \omega_{1-\alpha}(t)-\sum_{l=1}^{N_q}\varpi^l e^{-s^lt}\right| \leq \epsilon, \quad \forall t\in [\Delta t, T].
	\end{align*}
\end{lemma}

Next, the Caputo fractional derivative at the time point $t_{n-\theta}$ will be divided into two parts: an integral over $[0, t_{n-1}]$ (the historical part) and an integral over $[t_{n-1}, t_{n-\theta}]$ (the local part). The local part will be approximated directly via a linear interpolation and  the historical part will be evaluated by the SOE approximation given in Lemma \ref{FA}, that is
\begin{align}\nonumber
\mathcal{D}_{t}^{\alpha}u(t_{n-\theta})&\approx \int_{t_{n-1}}^{t_{n-\theta}}\varpi'_n(s)(\Pi_{1,n}u)'(s)\zd s + \int_{0}^{t_{n-1}}{\sum_{l=1}^{N_q}\varpi^l e^{-s^l(t_{n-\theta}-s)}u'(s)\zd s}\\\label{fastCaputo}
&= a_0^{(n)}\nabla_\tau u^n + \sum_{l=1}^{N_q}\varpi^l \mathcal{Q}^l(t_{n-1}), \quad n\geq 1,
\end{align}
where ($1\leq k\leq n$)
\begin{align}\label{ank}
	&a_{n-k}^{(n)}:=\frac1{\tau_k}\int_{t_{k-1}}^{\min\{t_k,t_{n-\theta}\}}\omega_{1-\alpha}(t_{n-\theta}-s)\zd s ,\\\nonumber
	&\mathcal{Q}^l(t_0):=0, \qquad\mathcal{Q}^l(t_k):=\int_{0}^{t_k}{e^{-s^l(t_{k+1-\theta}-s)}u'(s)\zd s}.
\end{align}
The quantity $\mathcal{Q}^l(t_k)$ can be approximated by using the quadratic interpolation and a recursive formula, i.e.,
\begin{align}\nonumber
\mathcal{Q}^l(t_k)&\approx \int_{0}^{t_{k-1}}{e^{-s^l(t_{k+1-\theta}-s)}u'(s)\zd s} + \int_{t_{k-1}}^{t_{k}}{e^{-s^l(t_{k+1-\theta}-s)}(\Pi_{2,k}u)'(s)\zd s}\\\label{Qapp}
&= e^{-s^l(\theta\tau_k+(1-\theta)\tau_{k+1})}\mathcal{Q}^l(t_{k-1})+a^{(k,l)}\nabla_\tau u^k+b^{(k,l)}(\rho_k\nabla_\tau u^{k+1}-\nabla_\tau u^k),
\end{align}
in which the positive coefficients $a^{(k,l)}$ and $b^{(k,l)}$ are respectively determined by
$$a^{(k,l)}:=\dfrac{1}{\tau_k}\int_{t_{k-1}}^{t_{k}}e^{-s^l(t_{k+1-\theta}-s)}\zd s, \quad b^{(k,l)}:=\dfrac{1}{\tau_k}\int_{t_{k-1}}^{t_{k}}e^{-s^l(t_{k+1-\theta}-s)}\frac{2(s-t_{k-1/2})}{\tau_k(\tau_k+\tau_{k+1})}\zd s.$$
Thus, from \eqref{fastCaputo}--\eqref{Qapp}, the fast Alikhanov formula is presented as
\begin{align}\label{fastAlik}
(\mathcal{D}_{\tau}^{\alpha}u)^{n-\theta} = a_0^{(n)}\nabla_\tau u^n + \sum_{l=1}^{N_q}\varpi^l \mathcal{Q}^l(t_{n-1}), \quad n\geq 1.
\end{align}
It can be observed that the average storage of the approximation \eqref{fastAlik} is ${\cal O}(N_{\exp})$ instead of ${\cal O}(N)$, where the later one is generated from classical Alikhanov approximation, while computing the discrete Caputo derivative at the terminal point $t_N$. Thus the total computational cost of the corresponding numerical scheme with the SOE approximation will be far less than that of the standard schemes with classical Alikhanov approximation while $N$ is large.

One may notice that the discrete formula \eqref{fastAlik} has the following alternative form
\begin{align}\nonumber
(\mathcal{D}_{\tau}^{\alpha}u)^{n-\theta}&= \int_{t_{n-1}}^{t_{n-\theta}}\varpi'_n(s)(\Pi_{1,n}u)'(s)\zd s + \sum_{k=1}^{n-1}{\int_{t_{k-1}}^{t_{k}}{\sum_{l=1}^{N_q}\varpi^l e^{-s^l(t_{n-\theta}-s)}(\Pi_{2,k}u)'(s)\zd s}}\\\label{alternative}
&= a_0^{(n)}\nabla_\tau u^n + \sum_{k=1}^{n-1}\sum_{l=1}^{N_q}\varpi^l\left(c^{(k,l)}\nabla_\tau u^k+d^{(k,l)}(\rho_k\nabla_\tau u^{k+1}-\nabla_\tau u^k)\right),
\end{align}
where the discrete coefficients $c^{(k,l)}$ and $d^{(k,l)}$ are defined by
\begin{align}\label{ck}
c^{(k,l)}&:=\dfrac{1}{\tau_k}\int_{t_{k-1}}^{t_{k}}e^{-s^l(t_{n-\theta}-s)}\zd s, \\\label{dk} 
d^{(k,l)}&:=\int_{t_{k-1}}^{t_{k}}e^{-s^l(t_{n-\theta}-s)}\frac{2(s-t_{k-1/2})}{\tau_k(\tau_k+\tau_{k+1})}\zd s.
\end{align}
 Rearranging the terms in \eqref{alternative}, we obtain the compact form of \eqref{alternative}:
 \begin{align}\nonumber
	(\mathcal{D}_{\tau}^{\alpha}u)^{n-\theta} ={\sum_{k=1}^n} A_{n-k}^{(n)}{\nabla}_{\tau}u^k, \quad n\geq 1,
\end{align}
where the discrete convolution kernel ${A}_{n-k}^{(n)}$ are defined as follows: ${A}_0^{(1)} := a_0^{(1)}$ if n = 1 and, for $n\geq 2$,
\begin{align}\label{An-k}
{A}_{n-k}^{(n)}:=\left\{\begin{array}{ll}
\displaystyle a_0^{(n)}+\sum_{l=1}^{N_q}\varpi^l\rho_{n-1}d^{(n-1,l)},\quad & k=n,\\
\displaystyle\sum_{l=1}^{N_q}\varpi^l\left( \rho_{k-1}d^{(k-1,l)}+c^{(k,l)}-d^{(k,l)}\right) , & 2\leq k\leq n-1,\\
\displaystyle\sum_{l=1}^{N_q}\varpi^l(c^{(1,l)}-d^{(1,l)}),  & k=1.
\end{array}\right.
\end{align}
To analyze the proposed numerical scheme later, we need to show that above discrete convolution kernel $A_{n-k}^{(n)}$ fulfill two basic properties \cite{Adiscrete}, i.e.,\\
$\mathbf{A1.}$ There is a constant $\pi_A>0$ such that
$$ A_{n-k}^{(n)}\geq \frac{1}{\pi_A\tau_k}\int_{t_{k-1}}^{t_k}\omega_{1-\alpha}(t_n-s)ds  \quad~\mbox{ for }~  1\leq k\leq n\leq N;$$
$\mathbf{A2.}$ The discrete kernels are positive and monotone, that is,
$$A_{0}^{(n)}\geq A_{1}^{(n)}\geq A_{2}^{(n)}\geq \cdots \geq A_{n-1}^{(n)}>0  \quad~\mbox{ for }~  1\leq k\leq n\leq N.$$
We first derive some properties of the discrete coefficients $c^{(k,l)}$ and $d^{(k,l)}$.
\begin{lemma}\cite[Lemma 2.1]{Liao-2order}\label{bnk}
	For any function $q\in C^2([t_{k-1},t_k])$,
	\begin{align*}
		\int_{t_{k-1}}^{t_k}{(s-t_{k-1/2})q'(s)\zd s} &= -\int_{t_{k-1}}^{t_k}{(\widetilde{\Pi_{1,k}}q)(s)\zd s}\\
		&= \frac{1}{2}\int_{t_{k-1}}^{t_k}{(s-t_{k-1})(t_k-s)q''(s)\zd s}.
	\end{align*}
\end{lemma}
Applying Lemma \ref{bnk}, the definition \eqref{dk} of $d^{(k,l)}$ gives
\begin{align}\nonumber
	d^{(k,l)} &= \int_{t_{k-1}}^{t_{k}}e^{-s^l(t_{n-\theta}-s)}\frac{2(s-t_{k-1/2})}{\tau_k(\tau_k+\tau_{k+1})}\zd s\\\label{dkl}
	&= s^l\int_{t_{k-1}}^{t_{k}}e^{-s^l(t_{n-\theta}-s)}\frac{(s-t_{k-1})(t_k-s)}{\tau_k(\tau_k+\tau_{k+1})}\zd s,\quad  1\leq k\leq n-1.
\end{align}
Since $0<(s-t_{k-1})(t_k-s)<\tau_k^2/4$ ~for~ $t_{k-1}<s<t_k$, we have
\begin{align}\nonumber
d^{(k,l)} &\leq s^l\int_{t_{k-1}}^{t_{k}}e^{-s^l(t_{n-\theta}-s)}\frac{\tau_k^2}{4\tau_k(\tau_k+\tau_{k+1})}\zd s\\\label{dchu}
&= \frac{s^l\rho_k}{4(1+\rho_k)}\int_{t_{k-1}}^{t_{k}}e^{-s^l(t_{n-\theta}-s)}\zd s,\quad  1\leq k\leq n-1.
\end{align}
For simplicity of presentation, we let
\begin{align}
I^{(k,l)} := s^l\int_{t_{k-1}}^{t_{k}}{\frac{t_k-s}{\tau_k}e^{-s^l(t_{n-\theta}-s)}\zd s}, ~ J^{(k,l)} := s^l\int_{t_{k-1}}^{t_{k}}{\frac{s-t_{k-1}}{\tau_k}e^{-s^l(t_{n-\theta}-s)}\zd s},\quad  1\leq k\leq n-1.
\end{align}
\begin{lemma}\label{IJ}
	For $1\leq k\leq n-1$, the positive coefficients $d^{(k,l)}$ in \eqref{dk}  satisfy\\
	(i) $I^{(k,l)}\geq \dfrac{1+\rho_k}{\rho_k}d^{(k,l)};$\qquad (ii) $J^{(k,l)}\geq \dfrac{2(1+\rho_k)}{\rho_k}d^{(k,l)};$\qquad (iii) $J^{(k+1,l)}\geq \dfrac{1}{\rho_k}J^{(k,l)}.$
\end{lemma}
\begin{proof}
	The alternative definition \eqref{dkl} of $d^{(k,l)}$ gives the result (i) directly since $0<s-t_{k-1}<\tau_k$ for $s\in (t_{k-1},t_k)$. Since $e^{-s^l(t_{n-\theta}-s)}>0$ for $0<s<t_{n-\theta}$, we apply Lemma \ref{bnk} to find
	$$s^l\int_{t_{k-1}}^{t_{k}}{\left( \frac{s-t_{k-1}}{\tau_k}-\dfrac{1}{2}\right) e^{-s^l(t_{n-\theta}-s)}\zd s}=\dfrac{(s^l)^2}{2\tau_k}\int_{t_{k-1}}^{t_{k}}(s-t_{k-1})(t_k-s)e^{-s^l(t_{n-\theta}-s)}\zd s>0,$$
	and then $J^{(k,l)}>\frac{s^l}{2}\int_{t_{k-1}}^{t_{k}}{e^{-s^l(t_{n-\theta}-s)}\zd s}$ for $1\leq k\leq n-1$.  So the inequality (ii) follows immediately from \eqref{dchu}. We now introduce an auxiliary function
	$$G_k(z) := \dfrac{s^l}{\tau_k}\int_{t_{k-1}}^{t_{k-1}+z\tau_k}(s-t_{k-1})e^{-s^l(t_{n-\theta}-s)}\zd s, \quad  1\leq k\leq n-1,\quad z\in [0,1],$$
	with its first-order derivative $G_k'(z)=zs^l\tau_ke^{-s^l(t_{n-\theta}-(t_{k-1}+z\tau_k))}$ for $1\leq k\leq n-1$. By using the Cauchy differential mean-value theorem, there exist $\xi \in (0,1)$ such that
	$$ \frac{J^{(k+1,l)}}{J^{(k,l)}}=\frac{G_{k+1}(1)}{G_k(1)}=\frac{G_{k+1}(1)-G_{k+1}(0)}{G_k(1)-G_k(0)}=\frac{G_{k+1}'(\xi)}{G_k'(\xi)}=\frac{\tau_{k+1}e^{-s^l(t_{n-\theta}-(t_{k}+\xi\tau_{k+1}))}}{\tau_ke^{-s^l(t_{n-\theta}-(t_{k-1}+\xi\tau_k))}}\geq \frac{1}{\rho_k},$$
which yields the inequality (iii).
\end{proof}
\begin{lemma}\label{plus}
	The positive coefficients $c^{(k,l)}$ in \eqref{ck} satisfy
	$$c^{(k+1,l)}-c^{(k,l)}=I^{(k+1,l)}+J^{(k,l)}, \quad 1\leq k\leq n-1 ~(2\leq n\leq N).$$
\end{lemma}
\begin{proof}
	For fixed $n~ (2\leq n\leq N)$, from the definition \eqref{ck}, we exchange the order of integration to find (for $1\leq k\leq n-1$)
\begin{align}\nonumber
c^{(k+1,l)}-e^{-s^l(t_{n-\theta}-t_k)}=&\int_{t_{k}}^{t_{k+1}}\frac{e^{-s^l(t_{n-\theta}-s)}-e^{-s^l(t_{n-\theta}-t_k)}}{\tau_{k+1}}\zd s\\\label{cI}
=&\int_{t_{k}}^{t_{k+1}}\int_{t_k}^{s}\frac{s^le^{-s^l(t_{n-\theta}-y)}}{\tau_{k+1}}\zd y\zd s=I^{(k+1,l)}.
	\end{align}
	Similarly, for $1\leq k\leq n-1~(2\leq n\leq N)$,
	\begin{align}\label{cJ}
		c^{(k,l)}-e^{-s^l(t_{n-\theta}-t_k)}=\int_{t_{k-1}}^{t_{k}}\frac{e^{-s^l(t_{n-\theta}-s)}-e^{-s^l(t_{n-\theta}-t_k)}}{\tau_{k}}\zd s=-J^{(k,l)}.
	\end{align}
	The proof is complete.
\end{proof}
\begin{lemma}\label{cd}
	If $\mathbf{M1}$ holds, for $1\leq k\leq n-1~(2\leq n\leq N)$, the positive coefficients $c^{(k,l)}$ in \eqref{ck} satisfy
	\begin{align*}
	c^{(k+1,l)}-c^{(k,l)}>\left\{\begin{array}{ll}
	d^{(2,l)},\quad & k=1,\\
	d^{(k+1,l)}+\rho_{k-1}d^{(k-1,l)}  & 2\leq k\leq n-1,
	\end{array}\right.
	\end{align*}
\end{lemma}
\begin{proof}
	By using Lemma \ref{IJ} (ii) and (iii),
	$$\frac{\rho_{k-1}^3}{2(1+\rho_{k-1})}J^{(k,l)}\geq \frac{\rho_{k-1}^2}{2(1+\rho_{k-1})}J^{(k-1,l)}\geq \rho_{k-1}d^{(k-1,l)},\quad~2\leq k\leq n-1.$$
Since $2+2y-y^3\geq 9/64$ for $y\in [0,7/4]$, the mesh assumption $\mathbf{M1}$ leads to
	\begin{align}\label{Jd}
		J^{(k,l)}=\frac{\rho_{k-1}^3}{2(1+\rho_{k-1})}J^{(k,l)}+\frac{2+2\rho_{k-1}-\rho_{k-1}^3}{2(1+\rho_{k-1})}J^{(k,l)}>\rho_{k-1}d^{(k-1,l)},\quad~2\leq k\leq n-1.
	\end{align}
	Hence, it follows from Lemma \ref{IJ} (i) and Lemma \ref{plus} that
		\begin{align*}
	c^{(k+1,l)}-c^{(k,l)}=I^{(k+1,l)}+J^{(k,l)}>\left\{\begin{array}{ll}
	d^{(2,l)},\quad & k=1,\\
	d^{(k+1,l)}+\rho_{k-1}d^{(k-1,l)}  & 2\leq k\leq n-1,
	\end{array}\right.
	\end{align*}
	The proof is complete.
\end{proof}
\begin{lemma}\label{special}
	If the tolerance error $\epsilon$ of the SOE approximation satisfies $\epsilon\leq\dfrac{\theta}{1-\alpha} \omega_{1-\alpha}(T)$, then the discrete coefficients $a_{n-k}^{(n)}$ of \eqref{ank} satisfies\\
(i) $\displaystyle a_0^{(2)}-\sum_{l=1}^{N_q}\varpi^lc^{(1,l)}\geq 0$;
~~(ii) $a_0^{(n)}-\sum_{l=1}^{N_q}\varpi^lc^{(n-1,l)}-\sum_{l=1}^{N_q}\varpi^l\rho_{n-2}d^{(n-2,l)}>0.$
\end{lemma}
\begin{proof}
Noticing $a_0^{(n)}=\frac{1-\theta}{1-\alpha}\omega_{1-\alpha}(t_{n-\theta}-t_{n-1})$ from \eqref{ank}, then we have
$$a_0^{(n)}-a_1^{(n)}\geq a_0^{(n)}-\omega_{1-\alpha}(t_{n-\theta}-t_{n-1})=\dfrac{\theta}{1-\alpha}\omega_{1-\alpha}(t_{n-\theta}-t_{n-1})\geq \dfrac{\theta}{1-\alpha} \omega_{1-\alpha}(T)\geq \epsilon.$$
	Therefore, Lemma \ref{FA} gives the result (i) directly since $\displaystyle a_0^{(2)}-\sum_{l=1}^{N_q}\varpi^lc^{(1,l)}\geq a_0^{(2)}-a_1^{(2)}-\epsilon\geq 0$. 
By \eqref{cJ} and \eqref{Jd}, we have
	$$\sum_{l=1}^{N_q}\varpi^l(c^{(n-1,l)}+\rho_{n-2}d^{(n-2,l)})<\sum_{l=1}^{N_q}\varpi^l(c^{(n-1,l)}+J^{(n-1,l)})=\sum_{l=1}^{N_q}\varpi^le^{-s^l(t_{n-\theta}-t_{n-1})},$$
then we apply Lemma \eqref{FA} to arrive that
	$$a_0^{(n)}-\sum_{l=1}^{N_q}\varpi^lc^{(n-1,l)}-\sum_{l=1}^{N_q}\varpi^l\rho_{n-2}d^{(n-2,l)}>\epsilon+\omega_{1-\alpha}(t_{n-\theta}-t_{n-1})-\sum_{l=1}^{N_q}\varpi^le^{-s^l(t_{n-\theta}-t_{n-1})}\geq 0.$$
	So the inequality (ii) is proved.
\end{proof}
	

We now verify that the coefficients  ${A}_{n-k}^{(n)}$ satisfy $\mathbf{A1}$ and ${\bf A2}$. Part (I) in the next lemma ensures that assumption $\mathbf{A2}$ is valid, while part (II) implies that assumption $\mathbf{A1}$ holds true with $\pi_A=11/4$.
\begin{lemma}
If the tolerance error $\epsilon$ of SOE approximation satisfies\\ $\epsilon\leq \min \{\frac{7}{11}\omega_{1-\alpha}(T), \frac{\theta}{1-\alpha} \omega_{1-\alpha}(T)\}$, then the discrete convolutional kernel ${A}_{n-k}^{(n)}$ in \eqref{An-k} satisfies\\
	    (I)\quad ${A}_{n-k-1}^{(n)}>{A}_{n-k}^{(n)}>0, \qquad 1\leq k\leq n-1$,\\
		(II)\quad ${A}_{n-k}^{(n)}\geq \dfrac{4}{11\tau_k} \int_{t_{k-1}}^{t_k}\omega_{1-\alpha}(t_n-s)ds, \qquad 1\leq k\leq n$.		
\end{lemma}

\begin{proof}
Recalling the definition \eqref{An-k}, it is not difficult to verify that\\
(1)~If $k=1$ for $n=2$,
\begin{align*}
A_0^{(2)}-A_1^{(2)} &= a_0^{(2)}+\sum_{l=1}^{N_q}\varpi^l\rho_{1}d^{(1,l)}-\sum_{l=1}^{N_q}\varpi^l(c^{(1,l)}-d^{(1,l)})\\
&=\sum_{l=1}^{N_q}\varpi^l(\rho_{1}+1)d^{(1,l)}+a_0^{(2)}-\sum_{l=1}^{N_q}\varpi^lc^{(1,l)};
\end{align*}
(2)~If $k=n-1$ for $n\geq 3$,
\begin{align*}
A_0^{(n)}-A_1^{(n)} &= a_0^{(n)}+\sum_{l=1}^{N_q}\varpi^l\rho_{n-1}d^{(n-1,l)}-\sum_{l=1}^{N_q}\varpi^l\left(\rho_{n-2}d^{(n-2,l)}+ c^{(n-1,l)}-d^{(n-1,l)}\right) \\
&=\sum_{l=1}^{N_q}\varpi^l(\rho_{n-1}+1)d^{(n-1,l)}+a_0^{(n)}-\sum_{l=1}^{N_q}\varpi^lc^{(n-1,l)}-\sum_{l=1}^{N_q}\varpi^l\rho_{n-2}d^{(n-2,l)};
\end{align*}
(3)~If $k=1$ for $n\geq 3$,
\begin{align*}
A_{n-2}^{(n)}-A_{n-1}^{(n)} &= \sum_{l=1}^{N_q}\varpi^l\left(\rho_{1}d^{(1,l)}+ c^{(2,l)}-d^{(2,l)}\right)-\sum_{l=1}^{N_q}\varpi^l(c^{(1,l)}-d^{(1,l)}) \\
&=\sum_{l=1}^{N_q}\varpi^l(\rho_{1}+1)d^{(1,l)}+\sum_{l=1}^{N_q}\varpi^l\left( c^{(2,l)}-c^{(1,l)}-d^{(2,l)}\right) ;
\end{align*}
(4)~If $2\leq k\leq n-2$ for $n\geq 4$,
\begin{align*}
A_{n-k-1}^{(n)}-A_{n-k}^{(n)} &= \sum_{l=1}^{N_q}\varpi^l\left(\rho_{k}d^{(k,l)}+ c^{(k+1,l)}-d^{(k+1,l)}\right)-\sum_{l=1}^{N_q}\varpi^l\left(\rho_{k-1}d^{(k-1,l)}+ c^{(k,l)}-d^{(k,l)}\right) \\
&=\sum_{l=1}^{N_q}\varpi^l(\rho_{k}+1)d^{(k,l)}+\sum_{l=1}^{N_q}\varpi^l\left( c^{(k+1,l)}-c^{(k,l)}-d^{(k+1,l)}-\rho_{k-1}d^{(k-1,l)}\right) .
\end{align*}
Hence the claimed inequality in the part(I) follows from Lemma \ref{cd} and Lemma \ref{special} directly.

According to the definitions \eqref{An-k} and \eqref{ank}, the inequality in part (II) holds obviously while $k=n$.
 Under the assumption $\mathbf{M1}$, and by using Lemma \ref{IJ} (i) and \eqref{cI}, one has
$$d^{(k,l)}\leq \dfrac{\rho_k}{\rho_k+1}I^{(k,l)}= \dfrac{\rho_k}{\rho_k+1}\left( c^{(k,l)}-e^{-s^l(t_{n-\theta}-t_{k-1})}\right) \leq \frac{7}{11}\left( c^{(k,l)}-e^{-s^l(t_{n-\theta}-t_{k-1})}\right).$$
Moreover, from  \eqref{ank}  and \eqref{ck}, and using Lemma \ref{FA}, we have 
$$\sum_{l=1}^{N_q}\varpi^lc^{(k,l)}\geq a_{n-k}^{(n)}-\epsilon,\quad \mbox{for}~1\leq k\leq n-1.$$
Thus, the lower bounds of $A_{n-k}^{(n)}$ for $1\leq k\leq n-1$ follow from Lemma \ref{FA} because the definition \eqref{An-k} implies that
\begin{align*}
	A_{n-k}^{(n)}&\geq \sum_{l=1}^{N_q}\varpi^l(c^{(k,l)}-d^{(k,l)})\geq \frac{4}{11}\sum_{l=1}^{N_q}\varpi^lc^{(k,l)}+\frac{7}{11}\sum_{l=1}^{N_q}\varpi^le^{-s^l(t_{n-\theta}-t_{k-1})}\\
	&\geq \frac{4}{11}a_{n-k}^{(n)}+\frac{7}{11}\omega_{1-\alpha}(t_{n-\theta}-t_{k-1})-\epsilon\geq \frac{4}{11}a_{n-k}^{(n)}+\frac{7}{11}\omega_{1-\alpha}(T)-\epsilon\\
	&\geq \frac{4}{11}a_{n-k}^{(n)}\geq \dfrac{4}{11\tau_k} \int_{t_{k-1}}^{t_k}\omega_{1-\alpha}(t_n-s)ds.
\end{align*}
The proof of part (II) is complete.
\end{proof}

\section{The fast and nonuniform high-order scheme}\label{NumericalScheme}

\subsection{The numerical scheme}
Let $u_i^n$ be the discrete approximation of solution $u(x_i,t_n)$ for $x_i\in \bar{\Omega}_h$, $0\leq n\leq N$. Considering the first equation in \eqref{GoverningEq} at the grid points $(x_i,t_{n-\theta})$, utilizing the fast nonuniform Alikhanov formula \eqref{fastAlik} and the spatial high-order approximation \eqref{high-order-app}, we can obtain
\begin{align}\label{sc1}
\displaystyle {\cal H}g_i^{n-\theta}=\left(a+\frac{h^2b^2}{12a}\right)\delta_x^2{u_i^{n-\theta}}+b\delta_{\hat{x}}u_i^{n-\theta}+R_i^{n-\theta} \quad \mbox{for}~~ x_i\in \Omega_h, ~1\leq n\leq N,
\end{align}
where
$$g_i^{n-\theta}:=({\cal D}_{\tau}^\alpha u_i)^{n-\theta}+cu_i^{n-\theta}-f(x_i,t_{n-\theta}),$$
and
$R_i^{n-\theta}=\mathcal{H}(R_{t1})_i^{n-\theta}+\mathcal{H}(R_{t2})_i^{n-\theta}+(R_{t3})_i^{n-\theta}+(R_s)_i^{n-\theta}$,
in which $(R_{s})_i^{n-\theta}=\mathcal{O}(h^4)$ (according to \eqref{high-order-app}), and
\begin{align*}
(R_{t1})_i^{n-\theta}=&(\mathcal{D}_{\tau}^{\alpha}u_i)^{n-\theta}-\mathcal{D}_{t}^{\alpha}u(x_i,t_{n-\theta}),\\
 (R_{t2})_i^{n-\theta}=&-c(R_{u})_i^{n-\theta},\\
(R_{t3})_i^{n-\theta}=&\left(a+\frac{h^2b^2}{12a}\right)\delta_x^2(R_{u})_i^{n-\theta}+b\delta_{\hat{x}}(R_{u})_i^{n-\theta}.
\end{align*}
with the error of the weighted time approximation at $t_{n-\theta}$ is given as
$$(R_{u})_i^{n-\theta}=u(x_i,t_{n-\theta})-[\theta u(x_i,t_{n-1})+(1-\theta) u(x_i,t_n)].$$
Based on the regularity assumption \eqref{regularity1} and the mesh condition $\mathbf{M1}$, and referring to \cite[Lemma 3.6 and Lemma 3.8]{Liao-2order}, we can obtain that
\begin{align}\label{PR1}
&\displaystyle\sum_{k=1}^nP_{n-k}^{(n)}|(R_{t1})_i^{k-\theta}| \leq C\left( \tau_1^\sigma/\sigma+t_1^{\sigma-3}\tau_2^3+\frac{1}{1-\alpha}\mathop{\mbox{max}}\limits_{2\leq k \leq n}t_k^\alpha t_{k-1}^{\sigma-3}\tau_k^3/\tau_{k-1}^\alpha\right),\\\label{PR2}
&\displaystyle\sum_{k=1}^nP_{n-k}^{(n)}|(R_{u})_i^{k-\theta}| \leq C\left( \tau_1^{\sigma+\alpha}/\sigma+t_n^\alpha\mathop{\mbox{max}}\limits_{2\leq k \leq n}t_{k-1}^{\sigma-2}\tau_k^2\right),
\end{align}
where $P_{n-k}^{(n)}$ is called the discrete complementary convolution kernels which satisfying the basic rule:
$\sum_{j=k}^nP_{n-j}^{(n)}A_{j-k}^{(j)} \equiv 1$ for $1\leq k\leq n\leq N$.
Moreover, the complementary kernels are nonnegative and satisfy (\cite[Lemma 2.1]{Adiscrete})
\begin{align}\label{Psub}
\displaystyle \sum_{j=1}^{n}P_{n-j}^{(n)}\omega_{1+(m-1)\alpha}(t_j)\leq \frac{11}{4}\omega_{1+m\alpha}(t_n)  \quad\mbox{ for }~ m=0,1, ~\mbox{and}~1\leq n\leq N.
\end{align}
One may refer to \cite{Adiscrete,Liao-2order} for more details about the $P_{n-k}^{(n)}$ which is a crucial tool in the numerical analysis.
With \eqref{Psub} (for $m=1$), it is easy to show that
\begin{align}\label{PRs}
\max_{1\leq k \leq n}\sum_{j=1}^kP_{k-j}^{(k)}|(R_{s})_i^{j-\theta}|\leq C h^4.
\end{align}
Thus, combining \eqref{PR1}--\eqref{PR2} and \eqref{PRs}, it holds
\begin{align}\label{PR}
\max_{1\leq k\leq n}\sum_{k=1}^nP_{n-k}^{(n)}\|R^{k-\theta}\| \leq
C\left( \tau_1^\sigma/\sigma+\max_{2\leq k \leq n}t_k^\alpha t_{k-1}^{\sigma-3}\tau_k^3/\tau_{k-1}^\alpha+t_n^\alpha\max_{2\leq k \leq n}t_{k-1}^{\sigma-2}\tau_k^2+h^4\right).
\end{align}
Omitting the truncation errors in \eqref{sc1}, we get a fast and high-order nonuniform scheme for the problem \eqref{GoverningEq}:
\begin{align}\label{sc2}
\mathcal{H}(\mathcal{D}_t^{\alpha}u_i^{n-\theta}+cu_i^{n-\theta}-f_i^{n-\theta})=\left(a+\frac{h^2b^2}{12a}\right)\delta_x^2{u_i^{n-\theta}}+b\delta_{\hat{x}}u_i^{n-\theta}
, \quad x_i\in \Omega_h, ~1\leq n\leq N,
\end{align}
equipped with the initial condition $u_i^0=\varphi_i$ for $x_i\in \Omega_h$, and the boundary conditions $u_0^n = u_M^n =0$.

Denote ${\bf u}^n:=(u_1^n,u_2^n,\ldots,u_{M-1}^n)^T$, ${\bf f}^{n-\theta}=(f_1^{n-\theta},f_2^{n-\theta},\ldots,f_{M-1}^{n-\theta})^T$, and the matrices of the central differences
\begin{eqnarray}\nonumber
A:=\left[
    \begin{array}{ccccc}
       -2 &1 & & & \\
       1 &-2 &1 && \\
       & \ddots &\ddots&\ddots& \\
       && 1& -2& 1\\
       &&&1 &-2\\
    \end{array}
\right],
\quad
S:=\left[
    \begin{array}{ccccc}
       0 &1 &&&\\
       -1 &0 &1 &&\\
       & \ddots &\ddots  &\ddots &\\
       & & -1 &0 &1\\
       & &  &-1 &0\\
    \end{array}
\right].
\end{eqnarray}
We can rewrite the scheme \eqref{sc2} into the following equivalent matrix-vector equation:
\begin{align}\label{matrix-form}
H\mathcal{D}_t^{\alpha}{\bf u}^{n-\theta}+cH{\bf u}^{n-\theta}-H{\bf f}^{n-\theta}-\widehat{\bf f}^{n-\theta}= \left(\frac{a}{h^2}+\frac{b^2}{12a}\right)A {\bf u}^{n-\theta}+\frac{b}{2h}S{\bf u}^{n-\theta},
\end{align}
where 
$$\widehat{\bf f}^{n-\theta}=\left[\left(\dfrac{1}{12}-\dfrac{hb}{24a}\right){f}^{n-\theta}_0,0,\cdots,0,\left(\dfrac{1}{12}+\dfrac{hb}{24a}\right){f}^{n-\theta}_M\right]^T,$$
 and $H=\dfrac{1}{12}A+\dfrac{hb}{24a}S+I$ with $I$ being the $(M-1)$-dimensional unit matrix.

\subsection{Stability and convergence}\label{Stability-convergence}

\begin{lemma}\cite[Corollary 2.3]{Liao-2order}\label{Du}
Under the condition $\mathbf{M1}$, the discrete Caputo formula satisfies
$$\displaystyle\langle(\mathcal{D}_{\tau}^{\alpha}v)^{n-\theta},v^{n-\theta}\rangle\geq \frac{1}{2}\sum_{k=1}^{n}A_{n-k}^{(n)}\nabla_\tau(\|v^k\|^2),   \quad~\mbox{ for }~ 1\leq n\leq N.$$
\end{lemma}

According to \cite[Theorem 3.1 and Remark 1]{Adiscrete}, we have the following lemma.
\begin{lemma}\label{FGi} 
Let the assumptions $\mathbf{A1}$ and ${\bf A2}$ hold, and let $(\xi^n)_{n=1}^N$ be a given nonnegative sequence. Then, for any nonnegative sequence $(v^k)_{k=0}^N$ such that 
$$\sum_{k=1}^n A_{n-k}^{(n)} \nabla_\tau (v^k)^2\leq \sum_{k=1}^n\lambda_{n-k} (v^{k-\theta})^2+v^{n-\theta}\xi^n,\quad \mbox{for}~1\leq n\leq N, $$
it holds that
$$v^n\leq v^0+\mathop{\max}\limits_{1\leq k \leq n}\sum_{j=1}^{k}P_{k-j}^{(k)}\xi^j\leq v^0+\dfrac{11}{4}\Gamma(1-\alpha)\max\limits_{1\leq j \leq n}\{t_j^\alpha\xi^j\},\quad \mbox{ for }~ 1\leq n\leq N.$$
\end{lemma}

To show the stability and convergence of the proposed scheme, we first discuss some main properties of the matrices in \eqref{matrix-form}.
\begin{lemma}(\cite{Laub})\label{numerical-range}
Let symmetric matrix ${\cal M}\in {\mathbb R}^{m\times m}$ with eigenvalues $\lambda_1\geq \lambda_2\geq\cdots\geq \lambda_m$. Then for all ${\bf w}\in {\mathbb R}^{m\times 1}$,
$$\lambda_m{\bf w}^T{\bf w}\leq {\bf w}^T{\cal M}{\bf w}\leq \lambda_1{\bf w}^T{\bf w}.$$
\end{lemma}

\begin{lemma}\label{HTH}
	The matrix $H^TH$ satisfies $\frac5{12}{\bf w}^T{\bf w}\leq{\bf w}^TH^TH{\bf w}\leq {\bf w}^T{\bf w}$ for any vector ${\bf w}$.
\end{lemma}
\begin{proof}
Denote
\begin{eqnarray}\nonumber
B:=\left[
\begin{array}{cccccc}
10 &1 & & & \\
1 &10 &1 && \\
& \ddots &\ddots&\ddots& \\
&& 1& 10& 1\\
&&&1 &10\\
\end{array}
\right].
\end{eqnarray}
Then $H=\dfrac{1}{12}A+\dfrac{b}{24a}hS+I=\dfrac{1}{12}B++\dfrac{b}{24a}hS$, and
\begin{align*}
{H}^TH=&\left(\dfrac{1}{12}B+\dfrac{b}{24a}hS^T\right) \left(\dfrac{1}{12}B+\dfrac{b}{24a}hS\right)\\
=&\displaystyle\frac{1}{144}B^2+\frac{b}{288a}hBS+\frac{b}{288a}hS^TB+\frac{b^2}{576a^2}h^2S^TS\\
=&\displaystyle\frac{1}{144}\left[ B^2+\frac{b}{2a}h(BS+S^TB)+\frac{b^2}{4a^2}h^2S^TS\right].
\end{align*}
Noticing that
\begin{eqnarray}\label{BS/ST}
BS+S^TB=\left[
\begin{array}{ccccc}
-2 & & & & \\
&0 & & & \\
& &\ddots & & \\
&&&0 & \\
&&&&2 \\
\end{array}
\right],
\quad
S^TS=\left[
\begin{array}{cccccc}
1 &0 &-1 & & & \\
0 &2 &0 &-1 & & \\
-1 &0 &2 &0 &-1 & \\
& \ddots &\ddots &\ddots &\ddots & \\
&&-1 &0 &2 &0 \\
&&&-1 &0 &1 \\
\end{array}
\right],
\end{eqnarray}
therefore,
\begin{eqnarray}\nonumber
H^TH=\frac{1}{144}\left[
\begin{array}{cccccc}
c_2 &20 &1-c_1h^2 & & & \\
20 &102+2c_1h^2 &20 &1-c_1h^2 & & \\
1-c_1h^2 &20 &102+2c_1h^2 &20 &1-c_1h^2 & \\
& \ddots &\ddots &\ddots &\ddots & \\
&&1-c_1h^2 &20 &102+2c_1h^2 &20 \\
&&&1-c_1h^2 &20 &c_3 \\
\end{array}
\right],
\end{eqnarray}
where $c_1=\dfrac{b^2}{4a^2}$, $c_2=101-\dfrac{b}{a}h+c_1h^2$, $c_3=101+\dfrac{b}{a}h+c_1h^2$. It is easy to check that $H^TH$ is diagonally dominant.

Next, we divide $H^TH$ into two parts to study its numerical range:
\begin{align*}
&H^TH=\frac{1}{144}\left[
\begin{array}{cccccc}
21-c_1h^2 &20 &1-c_1h^2 & & & \\
20 &41-c_1h^2 &20 &1-c_1h^2 & & \\
1-c_1h^2 &20 &42-2c_1h^2 &20 &1-c_1h^2 & \\
& \ddots &\ddots &\ddots &\ddots & \\
&&1-c_1h^2 &20 &41-c_1h^2 &20 \\
&&&1-c_1h^2 &20 &21-c_1h^2 \\
\end{array}
\right]\\
&{\small+\frac{1}{144}\left[
\begin{array}{ccccccc}
80-\dfrac{b}{a}h+2c_1h^2 & & & & &&\\
&61+3c_1h^2 & & & &&\\
& &60+4c_1h^2 & & &&\\
&&&\ddots & && \\
&&&&60+4c_1h^2 && \\
&&&&&61+3c_1h^2& \\
&&&&&&80+\dfrac{b}{a}h+2c_1h^2 \\
\end{array}
\right] }\\
&:=H_1+H_2,
\end{align*}
For small $h$, by using the Gershgorin\textquoteright s circle theorem, it is easy to know that $\lambda_{\min}(H_1)\geq 0$ and $\lambda_{\min}(H_2)\geq 60/144=5/12$. By Lemma \ref{numerical-range}, for any vector ${\bf w}$, it holds
$${\bf w}^TH^TH{\bf w}={\bf w}^TH_1{\bf w}+{\bf w}^TH_2{\bf w}\geq \frac5{12}{\bf w}^T{\bf w}.$$
Similarly, we have the decomposition
\begin{align*}
H^TH=&\frac{1}{144}\left[
\begin{array}{cccccc}
-21+c_1h^2 &20 &1-c_1h^2 & & & \\
20 &-41+c_1h^2 &20 &1-c_1h^2 & & \\
1-c_1h^2 &20 &-42+2c_1h^2 &20 &1-c_1h^2 & \\
& \ddots &\ddots &\ddots &\ddots & \\
&&1-c_1h^2 &20 &-41+c_1h^2 &20 \\
&&&1-c_1h^2 &20 &-21+c_1h^2 \\
\end{array}
\right]\\
&+\frac{1}{144}\left[
\begin{array}{ccccccc}
122-\dfrac{hb}{a} & & & & &&\\
&143+c_1h^2 & & & &&\\
& &144 & & &&\\
&&&\ddots & && \\
&&&&144 && \\
&&&&&143+c_1h^2 & \\
&&&&&&122+\dfrac{hb}{a} \\
\end{array}
\right] \\
:=&H_3+H_4.
\end{align*}
The Gershgorin\textquoteright s circle theorem gives $\lambda_{\max}(H_3)\leq0$ and $\lambda_{\max}(H_4)\leq 144/144=1$ for small $h$, which leads to
$${\bf w}^TH^TH{\bf w}={\bf w}^TH_3{\bf w}+{\bf w}^TH_4{\bf w}\leq {\bf w}^T{\bf w}.$$
\end{proof}

\begin{lemma}\label{HA-HS}
The matrices $H^TA+AH$ and $\dfrac{a}{h^2}(H^TA+AH)+\dfrac{b}{2h}(H^TS+S^TH)$ are negative semi-definite.
\end{lemma}
\begin{proof}
Straightforward computations show that
%
\begin{align*}
{H}^TA+AH=&\left(\dfrac{1}{12}B+\dfrac{b}{24a}hS^T\right) A+A\left(\dfrac{1}{12}B+\dfrac{b}{24a}hS\right)\\
=&\frac{1}{6}\left[
\begin{array}{cccccc}
-19-\dfrac{b}{2a}h &8 &1 & & & \\
8 &-18 &8 &1 & & \\
1 &8 &-18 &8 &1 & \\
& \ddots &\ddots &\ddots &\ddots & \\
&&1 &8 &-18 &8 \\
&&&1 &8 &-19+\dfrac{b}{2a}h \\
\end{array}
\right].
\end{align*}
So ${H}^TA+AH$ is diagonally dominant with small $h$. By similar arguments to the proof of Lemma \ref{HTH}, we can check that $\lambda_{\max}({H}^TA+AH)\leq 0$ which leads to first part of the desired result.
%

For the second part, we notice that
\begin{align*}
&\frac{a}{h^2}(H^TA+AH)+\frac{b}{2h}(H^TS+S^TH)\\
=&\frac{1}{6h^2}\left[
\begin{array}{cccccc}
c_4 &8a &a-ac_1h^2 & & & \\
8a &-18a+2ac_1h^2 &8a &a-ac_1h^2 & & \\
a-ac_1h^2 &8a &-18a+2ac_1h^2 &8a &a-ac_1h^2 & \\
& \ddots &\ddots &\ddots &\ddots & \\
&&a-ac_1h^2 &8a &-18a+2ac_1h^2&8a \\
&&&a-ac_1h^2 &8a & c_5 \\
\end{array}
\right].
\end{align*}
where $c_4=-19a-bh+ac_1h^2, c_5=-19a+bh+ac_1h^2$. One may easy to find that the above matrix is also diagonally dominant. Similarly, it is negative semi-definite.
\end{proof}

We are now ready to display the stability and convergence of our proposed scheme.
\begin{theorem}\label{Stability} (Stability)
If the assumptions $\mathbf{A1}$-${\bf A2}$ hold, then the numerical scheme \eqref{sc2} is stable and satisfies
\begin{align*}
\|u^n\|&\leq \sqrt{\dfrac{12}{5}}\left(\|u^0\|+2\mathop{\max}\limits_{1\leq k \leq n}\sum_{j=1}^{k}P_{k-j}^{(k)}(\|f^{n-\theta}\|+\|\widehat{f}^{n-\theta} \|)\right)\\
&\leq \sqrt{\dfrac{12}{5}}\left( \|u^0\|+\dfrac{11}{2}\Gamma(1-\alpha)\max\limits_{1\leq j \leq n}\{t_j^\alpha (\|f^{n-\theta}\|+\|\widehat{f}^{n-\theta} \|)\}\right) ,\quad 1\leq n\leq N.
\end{align*}
\end{theorem}
\begin{proof}
Multiplying both sides of \eqref{matrix-form} with $h({\bf u}^{n-\theta})^TH^T$, one has
\begin{equation}\label{stab1}
\begin{array}{l} \vspace{2mm}
\displaystyle h({\bf u}^{n-\theta})^TH^TH\mathcal{D}_\tau^{\alpha}{\bf u}^{n-\theta}+ch({\bf u}^{n-\theta})^TH^TH{\bf u}^{n-\theta}-h({\bf u}^{n-\theta})^TH^TH{\bf f}^{n-\theta}-h({\bf u}^{n-\theta})^TH^T\widehat{\bf f}^{n-\theta}\\
=\displaystyle\left(\frac{a}{h^2}+\frac{b^2}{12a}\right)h({\bf u}^{n-\theta})^TH^TA {\bf u}^{n-\theta}+\frac{b}{2h}h({\bf u}^{n-\theta})^TH^TS{\bf u}^{n-\theta}.
\end{array}
\end{equation}
According to Lemmas \ref{numerical-range} and \ref{HA-HS}, the right-hand side of \eqref{stab1} satisfies
\begin{align}\nonumber
&\left(\frac{a}{h^2}+\frac{b^2}{12a}\right)h({\bf u}^{n-\theta})^TH^TA {\bf u}^{n-\theta}+\frac{b}{2h}h({\bf u}^{n-\theta})^TH^TS{\bf u}^{n-\theta}\\\label{stab2}
=&
\frac{h}{2}\left(\frac{a}{h^2}+\frac{b^2}{12a}\right)({\bf u}^{n-\theta})^T(H^TA+AH){\bf u}^{n-\theta}+\frac{b}{4}({\bf u}^{n-\theta})^T(H^TS+S^TH){\bf u}^{n-\theta}\leq 0,
\end{align}
where the identity ${\bf u}^TW{\bf u}=\frac12 {\bf u}^T(W+W^T){\bf u}$ has been applied.

 From \eqref{stab1}--\eqref{stab2}, and utilizing Lemma \ref{HTH} and the Cauchy-Schwarz inequality, we get
\begin{align*}&h({\bf u}^{n-\theta})^TH^TH\displaystyle \mathcal{D}_t^{\alpha}{\bf u}^{n-\theta}\\
\leq& h({\bf u}^{n-\theta})^TH^TH{\bf f}^{n-\theta}+h({\bf u}^{n-\theta})^TH^T\widehat{\bf f}^{n-\theta}\\
\leq& h({\bf v}^{n-\theta})^T(H{\bf f}^{n-\theta})+h({\bf v}^{n-\theta})^T\widehat{\bf f}^{n-\theta}\\
\leq& \sqrt{h({\bf v}^{n-\theta})^T({\bf v}^{n-\theta})}\sqrt{h(H{\bf f}^{n-\theta})^T(H{\bf f}^{n-\theta})}+\sqrt{h({\bf v}^{n-\theta})^T({\bf v}^{n-\theta})}\sqrt{h(\widehat{\bf f}^{n-\theta})^T(\widehat{\bf f}^{n-\theta})}\\
\leq& \sqrt{h({\bf v}^{n-\theta})^T({\bf v}^{n-\theta})}\sqrt{h({\bf f}^{n-\theta})^T({\bf f}^{n-\theta})}+\sqrt{h({\bf v}^{n-\theta})^T({\bf v}^{n-\theta})}\sqrt{h(\widehat{\bf f}^{n-\theta})^T(\widehat{\bf f}^{n-\theta})},
\end{align*}
where ${\bf v}^{n-\theta}=H{\bf u}^{n-\theta}$, and it leads to (by taking $v_i^n={\cal H}u_i^n$)
\begin{align}\label{stb3}
\langle  (\mathcal{D}_\tau^{\alpha} v)^{n-\theta},v^{n-\theta}\rangle\leq \|v^{n-\theta} \|(\|f^{n-\theta} \|+\|\widehat{f}^{n-\theta} \|)\leq \left(\theta\|v^{n-1}\|+(1-\theta)\|v^n\|\right)(\|f^{n-\theta} \|+\|\widehat{f}^{n-\theta} \|).
\end{align}
Hence, it follows from Lemmas \ref{Du} and \ref{FGi}, and \eqref{stb3} that
\begin{align*}
\|v^n\|\leq& \|v^0\|+2\mathop{\max}\limits_{1\leq k \leq n}\sum_{j=1}^{k}P_{k-j}^{(k)}(\|f^{n-\theta}\|+\|\widehat{f}^{n-\theta} \|)\\
\leq& \|v^0\|+\dfrac{11}{2}\Gamma(1-\alpha)\max\limits_{1\leq j \leq n}\{t_j^\alpha (\|f^{n-\theta}\|+\|\widehat{f}^{n-\theta} \|)\},\quad\mbox{ for }~ 1\leq n\leq N.
\end{align*}
Finally, the claimed result can be achieved provided the following bound arising from Lemma \ref{HTH}:
$$\frac{5}{12}\|u^n\|^2 \leq \|v^n\|^2\leq\|u^n\|^2.$$
\end{proof}

\begin{theorem}\label{Convergence} (Convergence)
Let $e_i^n=u(x_i,t_n)-u_i^n$, if $\mathbf{M1}$, $\mathbf{M2}$ and the regularity assumptions \eqref{regularity1}--\eqref{regularity2} hold, then the proposed scheme \eqref{sc2} is convergent with
\begin{align*}
 \|e^n\| \leq C(\tau^{\min\{\gamma\sigma,2\}}+h^4),  \quad   1\leq n\leq N.
\end{align*}
\end{theorem}

\begin{proof}
We have the following error equation
\begin{align}\label{Errorequation}
 {\cal H} \mathcal{D}_t^{\alpha}e_i^{n-\theta}=\left(a+\frac{h^2b^2}{12a}\right)\delta_x^2{e_i^{n-\theta}}+b\delta_{\hat{x}}e_i^{n-\theta}-c{\cal H}e_i^{n-\theta}+R_i^{n-\theta},\quad x_i\in \Omega_h, ~1\leq n\leq N.
\end{align}
By using analogous derivations to the proof of Theorem \ref{Stability}, and notice \eqref{PR}, we can get
\begin{align*}
\|e^n\| \leq& \sqrt{\dfrac{12}{5}}\left(\|e^0\|+2\max_{1\leq k \leq n}\sum_{j=1}^kP_{k-j}^{(k)}\|R^{j-\theta}\|\right)\\
\leq& C\left({\tau_1^\sigma}+\mathop{\max}\limits_{2\leq k \leq n}t_k^\alpha t_{k-1}^{\sigma-3}\frac{\tau_k^3}{\tau_{k-1}^\alpha}+\mathop{\max}\limits_{2\leq k \leq n}t_{k-1}^{\sigma-2}\tau_k^2 +h^4 \right).
\end{align*}
If the mesh assumption $\mathbf{M2}$ holds, we have $\tau_1\leq C{\tau}^{\gamma}$ and (for $2\leq k\leq n$) 
\begin{align*}
 &\displaystyle t_k^\alpha t_{k-1}^{\sigma-3}\tau_k^3/\tau_{k-1}^\alpha \leq C t_k^{\max\{0,\sigma-(3-\alpha)/\gamma\}}\tau^{\min\{2,\gamma\sigma \}}, \\
& \displaystyle t_{k-1}^{\sigma-2}\tau_k^2\leq C t_k^{\max\{0,\sigma-2/\gamma\}}\tau^{\min\{2,\gamma\sigma \}},
\end{align*}
see also \cite[eqs. (3.11) and (3.12)]{Liao-2order}. Thus, the claimed result follows immediately.
\end{proof}

\section{Numerical Implementations}\label{numerical}
 In this section, we carry out numerical experiments to illustrate our theoretical statements.
 As the experimental merit of the SOE technique has been well demonstrated in many previous works, e.g. \cite{Fast-evaluation,Yan-fast,LiaoYanZhang2018,LyuVongANM2020}, we mainly focus on the accuracy verification, one may refer to \cite{LyuVongANM2020,Yan-fast} for the computational advantage of the fast Alikhanov approximation comparing to the classical approximation while solving the time-fractional partial differential equations.
 In our computations, the special domain is divided uniformly into $M$ subintervals and the time interval is divided by a general nonuniform grid with $N$ parts.

\begin{example}\label{ex1}
We consider that the problem \eqref{GoverningEq} with $\Omega=(0,1)$, $T=1$, the initial condition $u(x,0)= x^3(1-x)^3$, and the source term 
$$f(x,t)=x^3(1-x)^3\left( \frac{t^{1-\alpha}}{\Gamma(2-\alpha)}+\Gamma(\alpha+1)\right)-\left( a\frac{\partial^{2}u}{\partial x^{2}}+b\frac{\partial u}{\partial x}-cu\right)(t^\alpha+t+1)$$
is chosen to such that the exact solution is $u(x,t)=x^3(1-x)^3(t^\alpha+t+1)$.
Here we take $a=0.5$, $b=-0.45$ and $c=0.05$.
\end{example}

The numerical results of the proposed scheme \eqref{sc2} in solving Example 1 are recorded in Tables \ref{table1} and \ref{table2}. In each run, the discrete $L^2$-norm solution error
$$E_2(M,N)=\mathop{\max}\limits_{1\leq n \leq N}\|u(\cdot,t_n)-u^n\|.$$
 The temporal and spatial rate of convergence is estimated respectively by 
 $$Rate_\tau=\log_2\left[ \dfrac{E_2(M,N)}{E_2(M,2N)}\right] \quad \mbox{and} \quad Rate_h=\log_2\left[ \dfrac{E_2(M,N)}{E_2(2M,N)}\right] .$$

To test the sharpness of our error estimate, we choose the graded time mesh $t_k=T(k/N)^\gamma$ with $\gamma=2/\sigma$, one may notice that $\sigma=\alpha$ for this example. From Tables \ref{table1} and \ref{table2}, it is clearly that the proposed numerical method is convergent with second-order accuracy in time and fourth-order accuracy in space, which agrees well with the theoretical statement.

\begin{table}[hbt!]
	\begin{center}
		\caption{Numerical temporal accuracy in temporal direction for fixed $M=1000$ and different $\alpha$.}\label{table1}
		\renewcommand{\arraystretch}{1.0}
		\def\temptablewidth{0.9\textwidth}
		{\rule{\temptablewidth}{0.9pt}}
		\begin{tabular*}{\temptablewidth}{@{\extracolsep{\fill}}ccccccc}
			$N$ &\multicolumn{2}{c}{$\alpha=0.5$}&\multicolumn{2}{c}{$\alpha=0.7$}
			&\multicolumn{2}{c}{$\alpha=0.9$}\\
			\cline{2-3}\cline{4-5}\cline{6-7}
			&$E_2(M,N)$ &$Rate_\tau$  &$E_2(M,N)$  &$Rate_\tau$  &$E_2(M,N)$  &$Rate_\tau$\\\hline
			$8$   &1.1597e-05    & $\ast$  &1.2056e-05     & $\ast$  &5.7101e-06     &$\ast$\\
			$16$   &2.9584e-06    &1.9709   &3.0508e-06     &1.9825   &1.4290e-06     &1.9985 \\
			$32$  &7.5167e-07    &1.9766   &7.7019e-07     &1.9859   &3.5783e-07     &1.9977 \\
			$64$  &1.9016e-07    &1.9829   &1.9400e-07     &1.9892   &8.9585e-08     &1.9979 \\
			$128$  &4.7827e-08    &1.9913   &4.8775e-08     &1.9918   &2.2423e-08     &1.9983 \\
		\end{tabular*}
		{\rule{\temptablewidth}{0.9pt}}
	\end{center}
\end{table}

\begin{table}[hbt!]
	\begin{center}
		\caption{Numerical spatial accuracy in spatial direction for fixed $N=2000$ and different $\alpha$.}\label{table2}
		\renewcommand{\arraystretch}{1.0}
		\def\temptablewidth{0.9\textwidth}
		{\rule{\temptablewidth}{0.9pt}}
		\begin{tabular*}{\temptablewidth}{@{\extracolsep{\fill}}ccccccc}
			$M$ &\multicolumn{2}{c}{$\alpha=0.5$}&\multicolumn{2}{c}{$\alpha=0.7$}
			&\multicolumn{2}{c}{$\alpha=0.9$}\\
			\cline{2-3}\cline{4-5}\cline{6-7}
			&$E_2(M,N)$ &$Rate_h$  &$E_2(M,N)$  &$Rate_h$  &$E_2(M,N)$  &$Rate_h$\\\hline
			$4$   &2.7475e-03    & $\ast$  &2.7658e-03     & $\ast$  &2.7897e-03     &$\ast$\\
			$8$   &1.7422e-04    & 3.9791  &1.7508e-04     & 3.9816  &1.7659e-04     &3.9816\\
			$16$   &1.1220e-05    &3.9568   &1.0975e-05     &3.9957   &1.1067e-05     &3.9961 \\
			$32$  &1.0055e-06    &3.4800   &6.8963e-07     &3.9923   &6.9217e-07     &3.9989 \\
		\end{tabular*}
		{\rule{\temptablewidth}{0.9pt}}
	\end{center}
\end{table}

\begin{example}\label{ex2}
We then consider the original model \eqref{eq1} with nonhomogeneous boundary conditions
\begin{equation*}
\begin{array}{l} \vspace{2mm}
\displaystyle {\frac{\partial^\alpha C}{\partial \zeta^\alpha}}+\frac{1}{2}{\varrho^{2}}S^{2}\frac{\partial^{2}C}{\partial S^{2}}+(r-D)S\frac{\partial C}{\partial S}-rC=0,\\
C(S,T)=(\ln S)^3+(\ln S)^2+1,\\
C(0,\zeta)=(T-\zeta+1)^2, \quad C(1,\zeta)=3(T-\zeta+1)^2.
\end{array}
\end{equation*}
Here the parameters are set as $T=1$, $(S_l,S_r)=(0,1)$, $\varrho=1$, $r=1$ and $D=0$.
\end{example}

To solve the Example \ref{ex2}, we rewrite it into the form of \eqref{GoverningEq} in order to apply the proposed scheme \eqref{sc2}.  It can be known by calculation that $a=0.5$, $b=0.5$, $c=0.05$,
$f(x,t)=(2b-c-2cx)(t+1)^2-(4x+2)\left( \frac{t^{1-\alpha}}{\Gamma(2-\alpha)}+\frac{t^{2-\alpha}}{\Gamma(3-\alpha)}\right)$ and $\varphi(x)=x^3+x^2-2x$.

The numerical results of Example 2 are listed in Tables \ref{table3} and \ref{table4}. Since there is no exact solution for this example, we take the approximate errors ${\tilde E}_2(M,N)=\|u_M^{\tilde{N}}-u_M^N\|$ and ${\hat E}_2(M,N)=\|u_{\hat{M}}^{N}-u_M^N\|$, where $u_M^N$ is the numerical solution with mesh nodes $N$,$M$, and $u_M^{\tilde{N}}$ and $u_{\hat{M}}^{N}$ are the numerical solutions with relative dense meshes ($\tilde{N}=1024$ and $\hat{M}=1024$). 
The temporal and spatial convergence rates are calculated respectively by 
$$Rate_\tau=\log_2\left[ \dfrac{{\tilde E}_2(M,N)}{{\tilde E}_2(M,2N)}\right] \quad\mbox{and}\quad Rate_h=\log_2\left[ \dfrac{{\hat E}_2(M,N)}{{\hat E}_2(2M,N)}\right].$$

While solving the Example 2, we still choose the graded mesh $t_k=T(k/N)^\gamma$ with the grading parameter $\gamma=2/\alpha$.  The numerical results displayed in Tables \ref{table3} and \ref{table4} demonstrate that our proposed method works very well with the temporal second-order and spatial fourth-order convergence accuracy for the general time-fractional Black-Scholes equation.   

\begin{table}[hbt!]
 \begin{center}
 \caption{Numerical temporal accuracy in temporal direction with $M=1000$ and different $\alpha$.}\label{table3}
 \renewcommand{\arraystretch}{1.0}
 \def\temptablewidth{0.9\textwidth}
 {\rule{\temptablewidth}{0.9pt}}
 \begin{tabular*}{\temptablewidth}{@{\extracolsep{\fill}}ccccc}
$N$ &\multicolumn{2}{c}{$\alpha=0.7$}&\multicolumn{2}{c}{$\alpha=0.9$}\\
 \cline{2-3}\cline{4-5}
			&${\tilde E}_2(M,N)$ &$Rate_\tau$  &${\tilde E}_2(M,N)$  &$Rate_\tau$  \\\hline
       $4$   &2.4570e-02    & $\ast$  &1.7242e-02     & $\ast$  \\
       $8$   &7.0122e-03    &1.8089   &4.4057e-03     &1.9685  \\
       $16$  &1.8262e-03    &1.9411   &1.1134e-03     &1.9844   \\
       $32$  &4.4687e-04    &2.0309   &2.7911e-04     &1.9961   \\
       $64$  &9.3175e-05    &2.2618   &6.9612e-05     &2.0034   \\
\end{tabular*}
{\rule{\temptablewidth}{0.9pt}}
\end{center}
\end{table}

\begin{table}[hbt!]
	\begin{center}
		\caption{Numerical spatial accuracy in spatial direction with $N=2000$ and different $\alpha$.}\label{table4}
		\renewcommand{\arraystretch}{1.0}
		\def\temptablewidth{0.9\textwidth}
		{\rule{\temptablewidth}{0.9pt}}
		\begin{tabular*}{\temptablewidth}{@{\extracolsep{\fill}}ccccc}
			$M$&\multicolumn{2}{c}{$\alpha=0.7$}
			&\multicolumn{2}{c}{$\alpha=0.9$}\\
			\cline{2-3}\cline{4-5}
			&${\hat E}_2(M,N)$ &$Rate_h$  &${\hat E}_2(M,N)$  &$Rate_h$  \\\hline
			$4$   &3.6513e-04    & $\ast$  &3.3062e-04     & $\ast$  \\
			$8$   &2.3131e-05    & 3.9805  &2.0924e-05     & 3.9819  \\
			$16$   &1.4498e-06    &3.9959   &1.3112e-06     &3.9963   \\
			$32$  &9.0651e-08    &3.9994   &8.1957e-08     &3.9999   \\
			$64$  &5.6443e-09    &4.0055   &5.0804e-09     &4.0118   \\
		\end{tabular*}
		{\rule{\temptablewidth}{0.9pt}}
	\end{center}
\end{table}

\section{Conclusion}\label{conclusion}

We proposed a high-order and nonuniform finite difference method for solving the time-fractional Black-Scholes equation. The numerical method is constructed by combining the fast nonuniform Alikhanov formula and a spatial fourth-order average approximation. The unconditional stability and convergence of second-order in time and fourth-order in space are rigorously derived by energy method. Numerical examples are included and the results indicated that the proposed numerical method works very accurately.


\begin{thebibliography}{99}

\bibitem{Black and Scholes}
F. Black, M.S. Scholes, The pricing of options and corporate liabilities, J. Polit. Econ., 81 (1973), 637--654.

\bibitem{Castillo-Negrete}
A. Cartea, D. del Castillo-Negrete, Fractional diffusion models of option prices in markets with jumps, Physica A., 374 (2007), 749--763.

\bibitem{Cen-CMA2018}
Z. Cen, J. Huang, A. Xu, A. Le, Numerical approximation of a time-fractional Black-Scholes equation, Comput. Math. Appl., 75 (2018), 2874--2887.

\bibitem{ChenStynesJSC2019}
H. Chen,  M. Stynes, Error analysis of a second-order method on fitted meshes for a time-fractional diffusion problem, J. Sci. Comput., 79 (2019), 624--647.

\bibitem{Chen}
W. Chen, X. Xu, S. P. Zhu, Analytically pricing double barrier options based on a time-fractional Black-Scholes equation, Comput. Math. Appl., 69 (2015), 1407--1419.

\bibitem{Fast-evaluation}
S. Jiang, J. Zhang, Z. Qian, and Z. Zhang, Fast evaluation of the Caputo fractional derivative and its applications to fractional diffusion equations, Comm. Comput. Phys., 21 (2017), 650-678.

\bibitem{Jin-IMA2016}
B. Jin, R. Lazarov, Z. Zhou,  An analysis of the L1 scheme for the subdiffusion equation with nonsmooth data, IMA J. Numer. Anal.,  33 (2016), 197--221. 

\bibitem{Jumarie1}
G. Jumarie, Stock exchange fractional dynamics defined as fractional exponential growth driven by (usual) Gaussian white noise. Application to fractional Black-Scholes equations, Insurance Math. Econom., 42 (2008), 271--287.

\bibitem{Jumarie2}
G. Jumarie, Derivation and solutions of some fractional Black-Scholes equations in coarse-grained space and time. Application to Merton's optimal portfolio, Comput. Math. Appl., 59 (2010), 1142--1164.

\bibitem{KoptevaMC2019}
N. Kopteva, Error analysis of the L1 method on graded and uniform meshes for a fractional-derivative problem in two and three dimensions, Math. Comput., 88 (2019), 2135--2155.

\bibitem{Laub}
A. J. Laub, Matrix Analysis for Scientists and Engineers, SIAM, Philadelphia, 2005.

\bibitem{Liang1}
J. R. Liang, J. Wang, W. J. Zhang, W. Y. Qiu, F. Y. Ren, Option pricing of a bi-fractional Black-Merton-Scholes model with the Hurst exponent H in [1/2, 1], Appl. Math. Lett., 23 (2010), 859--863.

\bibitem{Liang2}
J. R. Liang, J. Wang, W. J. Zhang, W. Y. Qiu, F. Y. Ren, The solution to a bi-fractional Black-Scholes-Merton differential equation, Int. J. Pure Appl. Math., 58 (2010), 99--112.

\bibitem{Sharperror}
H. L. Liao, D. Li, J. Zhang, Sharp error estimate of nonuniform L1 formula for linear reaction-subdiffusion equations, SIAM J. Numer. Anal., 56 (2018), 1112--1133.

\bibitem{Adiscrete}
H. L. Liao, W. McLean, J. Zhang, A discrete Gr$\ddot{o}$nwall inequality with application to numerical schemes for fractional reaction-subdiffusion problems, SIAM J. Numer. Anal., 57 (2019), 218--237.

\bibitem{Liao-2order}
H. L. Liao, W. McLean, J. Zhang, A second-order scheme with nonuniform time steps for a linear reaction-subdiffusion problem, Commun. Comput. Phys., 30 (2021), 567--601.

\bibitem{second-fast}
H. L. Liao, T. Tang, T. Zhou, A second-order and nonuniform time-stepping maximum-principle preserving scheme for time-fractional Allen-Cahn equations, J. Comput. Phys., 414 (2020), 109473.

\bibitem{LiaoYanZhang2018}
H. L. Liao, Y. Yan,  J. Zhang, Unconditional convergence of a two-level linearized fast algorithm for semilinear subdiffusion equations, J. Sci. Comput., 80 (2019), 1--25.

\bibitem{LyuVongANM2020}
P. Lyu, Y. Liang, Z. Wang, A fast linearized finite difference method for the nonlinear multi-term time-fractional wave equation, Appl. Numer. Math., 151 (2020), 448--471.

\bibitem{LyuVong2020NA}
P.  Lyu, S. Vong, A fast linearized numerical method for nonlinear time-fractional diffusion equations, Numer. Algorithms, 87 (2021), 381--408.

\bibitem{LyuVong2020diffu-wave}
P.  Lyu, S. Vong, A symmetric fractional-order reduction method for direct nonuniform approximations of semilinear diffusion-wave equations, submitted. arXiv:2101.09678v3 [math.NA]

\bibitem{LyuVong2021wave-variable}
P.  Lyu, S. Vong, Second-order and nonuniform time-stepping schemes for time fractional evolution equations with time-space dependent coefficients, J. Sci. Comput., accepted. Also available on: arXiv:2102.09396v3 [math.NA]

\bibitem{Roul-AML}
P. Roul, A high accuracy numerical method and its convergence for time-fractional Black-Scholes equation governing European options, Appl. Numer. Math., 151 (2020), 472--493.

\bibitem{Staelen-CMA}
R. Staelen, A. Hendy, Numerically pricing double barrier options in a time-fractional Black-Scholes model, Comput. Math. Appl., 74 (2017), 1166--1175.

\bibitem{Stynes-SIAM2017}
M. Stynes, E. O'Riordan, J. L. Gracia, Error analysis of a finite difference method on graded meshes for a time-fractional diffusion equation, SIAM J. Numer. Anal., 55 (2017), 1057--1079.

\bibitem{WangANM2021}
Z. Wang, D. Cen, Y. Mo, Sharp error estimate of a compact L1-ADI scheme for the two-dimensional time-fractional integro-differential equation with singular kernels, Appl. Numer. Math., 159 (2021), 190--203.

\bibitem{Wyss}
W. Wyss, The fractional Black-Scholes equations, Fract. Calc. Appl. Anal., 3 (2000), 51--61.

\bibitem{Yan-fast}
Y. Yan, Z. Z. Sun, J. Zhang, Fast evaluation of the Caputo fractional derivative and its applications to fractional diffusion equations: A Second-order Scheme, Commun. Comput. Phys., 22 (2017), 1028--1048.

\bibitem{Zhang-Liu}
H. Zhang, F. Liu, I. Turner, Q. Yang, Numerical solution of the time fractional Black-Scholes model governing European options, Comput. Math. Appl., 71 (2016), 1772--1783.

\end{thebibliography}
\end{document}